\documentclass[11pt,a4paper]{article}
\usepackage[utf8]{inputenc}
\usepackage{amsmath}
\usepackage{amsfonts}
\usepackage{amssymb}
\usepackage{amsthm}
\usepackage{mathtools}
\usepackage[plain]{fullpage}
\usepackage[numbers,sort&compress,longnamesfirst]{natbib}
\usepackage{mdwlist}
\usepackage{doi}

\newcommand{\floor}[1]{\lfloor{#1}\rfloor}

\newcommand{\tw}{\textsf{\textup{tw}}}
\newcommand{\had}{\textsf{\textup{had}}}

\newcommand{\con}{\textsf{\textup{con}}}
\newcommand{\pw}{\textsf{\textup{pw}}}
\newcommand{\cw}{\textsf{\textup{cw}}}
\newcommand{\dd}{\textsf{\textup{d}}}
\newcommand{\zz}{\textsf{\textup{z}}}
\newcommand{\ee}{\textsf{\textup{e}}}
\newcommand{\bb}{\textsf{\textup{b}}}

\newtheorem{theorem}{Theorem}[section]
\newtheorem{lemma}[theorem]{Lemma}
\newtheorem{corollary}[theorem]{Corollary}
\newtheorem{proposition}[theorem]{Proposition}
\newtheorem{conjecture}[theorem]{Conjecture}

\theoremstyle{remark}
\newtheorem{claim}{Claim}[section]

\theoremstyle{definition}

\allowdisplaybreaks

\begin{document}
%Note for Tuesday - just go through the paper, add what needs to be added and change what needs to be changed, in order to hopefully make everything easier. perhaps leave the abstract to last, however. ALSO Note: check citations match up with previous papers, so I know that I've cited the correct thing.
\title{\bf The Treewidth of Line Graphs}

\author{Daniel~J.~Harvey\,\footnotemark[2] \qquad David~R.~Wood\,\footnotemark[2]}

\date{\today}

\maketitle

\footnotetext[2]{School of Mathematical Sciences, Monash University, Melbourne, Australia. \phantom{\footnotetext[2]} \texttt{\{Daniel.Harvey,David.Wood\}@monash.edu}. Supported by the Australian Research Council.}

\begin{abstract}
%The treewidth of a graph is a well-known graph invariant popularised by \citeauthor{RS-GraphMinors}. %The line graph $L(G)$ is a graph with vertex set $E(G)$ where two vertices are adjacent if they are incident as edges in $G$.
%The treewidth of \emph{line graphs} is key in recent work by \citeauthor{marxgrohe}. Here we show that determining the treewidth of a line graph of $G$ is equivalent to determining the minimum vertex congestion of an embedding of $G$ into a tree, and give sharp lower bounds %on the treewidth of $L(G)$ 
%in terms of both the minimum degree and average degree of $G$. These results are precise enough to exactly determine the treewidth of the line graph of a complete graph. We also improve the best known upper bound on the treewidth of a line graph. %$L(G)$.
The treewidth of a graph is an important invariant in structural and algorithmic graph theory. This paper studies the treewidth of \emph{line graphs}. We show that determining the treewidth of the line graph of a graph $G$ is equivalent to determining the minimum vertex congestion of an embedding of $G$ into a tree. Using this result, we prove sharp lower bounds in terms of both the minimum degree and average degree of $G$. These results are precise enough to exactly determine the treewidth of the line graph of a complete graph and other interesting examples. We also improve the best known upper bound on the treewidth of a line graph. Analogous results are proved for pathwidth.
\end{abstract}

\section{Introduction}
\label{sec:intro}
Treewidth is a graph parameter that measures how ``tree-like" a graph is. It is of fundamental importance in structural graph theory (especially in the graph minor theory of \citet{RS-GraphMinors}) and in algorithmic graph theory, since many NP-complete problems are solvable in polynomial time on graphs of bounded treewidth \citep{Bodlaender-AC93}. Let $\tw(G)$ denote the treewidth of a graph $G$ (defined below). This paper studies the treewidth of \emph{line graphs}.
For a graph $G$, the line graph $L(G)$ is the graph with vertex set $E(G)$ where two vertices are adjacent if and only if their corresponding edges are incident. 
%(Note that if we allow $G$ to have isolated vertices, the line graph $L(G)$ is the same graph as the line graph of $G$ with its isolated vertices removed, so this restriction is mostly for simplicity. It also ensures that $L(G)$ is not the empty graph. Finally, 
(We shall refer to vertices in the line graph as edges---vertices shall refer to the vertices of $G$ itself unless explicitly noted.)

As a concrete example, the treewidth of $L(K_n)$ is important in recent work by \citet{marxgrohe} and \citet{marx}. Specifically, \citet{marx} showed that if $\tw(G) \geq k$ then the lexicographic product of $G$ with $K_p$ contains the lexicographic product of $L(K_k)$ with $K_q$ as a minor (for choices of $p$ and $q$ depending on $|V(G)|$ and $k$). Motivated by this result, the authors determined the treewidth of $L(K_n)$ exactly \citep{mylinegraph-jgt}. The techniques used were extended to determine the treewidth of the line graph of a complete multipartite graph up to lower order terms, with an exact result when the complete multipartite graph is regular \citep{mythesis}. These results also extend to pathwidth (since the tree decompositions constructed have paths as the underlying trees.)

\paragraph{Lower Bounds.}
The following are two elementary lower bounds on $\tw(L(G))$. First, if $\Delta(G)$ is the maximum degree of $G$, then $\tw(L(G)) \geq \Delta(G)-1$ since the edges incident to a vertex in $G$ form a clique in $L(G)$. Second, given a minimum width tree decomposition of $L(G)$, replace each edge with both of its endpoints to obtain a tree decomposition of $G$. It follows that
\begin{equation}
\label{eq:easylb}
\tw(L(G)) \geq \tfrac{1}{2}(\tw(G)+1)-1.
\end{equation} 

We prove the following lower bound on $\tw(L(G))$ in terms of $\dd(G)^2$, where $\dd(G)$ is the average degree of $G$.
\begin{theorem}
\label{theorem:avgdegintro}
For every graph $G$ with average degree $\dd(G)$, $$\pw(L(G)) \geq \tw(L(G)) > \tfrac{1}{8}\dd(G)^2 + \tfrac{3}{4}\dd(G) - 2.$$
\end{theorem}

The bound in Theorem~\ref{theorem:avgdegintro} is within `$+1$' of optimal since we show that for all $k$ and $n$ there is an $n$-vertex graph $G$ with $\dd(G) \approx 2k$ and $\tw(L(G)) \leq \pw(L(G)) = \frac{1}{8}(2k)^2 + \frac{3}{4}(2k) -1$. All these results are proven in Section~\ref{section:avgdeg}.

We also prove a sharp lower bound in terms of $\delta(G)^2$, where $\delta(G)$ is the minimum degree of $G$. (The constants in Theorem~\ref{theorem:avgdegintro} and \ref{theorem:mindegintro} are such that, depending on the graph, either result could be stronger.)

\begin{theorem}
\label{theorem:mindegintro}
For every graph $G$ with minimum degree $\delta(G)$,
$$\pw(L(G)) \geq \tw(L(G)) \geq 
\begin{cases}
\frac{1}{4}\delta(G)^2+\delta(G) -1 &\text{ when $\delta(G)$ is even}\\
\frac{1}{4}\delta(G)^2+\delta(G)-\frac{5}{4} &\text{ when $\delta(G)$ is odd}.
\end{cases}$$
\end{theorem}

The bound in Theorem~\ref{theorem:mindegintro} is sharp since for all $n$ and $k$ we describe a graph $G$ with $n$ vertices and minimum degree $k$ such that $\pw(L(G))$ equals the bound in Theorem~\ref{theorem:mindegintro} when $n$ is even or $k$ is even, and is within `$+1$' when $n$ is odd and $k$ is odd. All these results are proven in Section~\ref{section:mindeg}.
%Conversely, for all integers $n,k$ such that $n > k$, there is a graph $G$ with minimum degree $2k$ such that $\pw(L(G))$ meets this lower bound exactly, and a graph $G'$ with minimum degree $2k-1$ such that $\pw(L(G'))$ meets this lower bound exactly when $n$ is even, or is within `$+1$' when $n$ is odd. Theorem~\ref{theorem:mindegintro} is proven in Section~\ref{section:mindeg}.

A weaker version of Theorem~\ref{theorem:mindegintro} first appeared in the first author's PhD thesis \citep{mythesis}. Theorems~\ref{theorem:avgdegintro} and \ref{theorem:mindegintro} are significant improvements for line graphs over the standard results that $\tw(G) \geq \delta(G)$ and $\tw(G) > \frac{1}{2}\dd(G)$ (which hold for all graphs), since $\delta(L(G)),\dd(L(G)),\delta(G)$ and $\dd(G)$ can be quite close. For example, $\delta(L(G))=\dd(L(G))=2\,\delta(G)-2=2\dd(G)-2$ when $G$ is regular.

In order to prove these results, we first show (in Section~\ref{sec:embed}) that constructing a tree decomposition of $L(G)$ is equivalent to determining a particular embedding of $G$ into a tree. This in turn allows us to prove a strong relationship between the treewidth of $L(G)$ and the vertex congestion of $G$, together with a similar relationship for the pathwidth of $L(G)$ and the vertex congestion of $G$ when embedded into a path. This second relationship is similar to a previous result relating $\pw(L(G))$ and cutwidth established by \citet{Golovach}. 

In Section~\ref{sec:altlow} we show that Theorems~\ref{theorem:avgdegintro} and~\ref{theorem:mindegintro} cannot be improved by replacing one of the $\dd(G)$ (or $\delta(G)$) terms by $\tw(G)$. 

Finally, we mention a related conjecture of Seymour, which was recently proved by \citet{immersion} using the theory of \emph{immersions}. It states that, given a graph $G$ with average degree $\dd(G)$, the Hadwiger number of $L(G)$ satisfies $\had(L(G)) \geq c\,\dd(G)^{\frac{3}{2}}$ for some constant $c>0$. They also show that the exponent $\frac{3}{2}$ is sharp due to the complete graph. Given that $\tw(L(G)) \geq \had(L(G))$, this gives a lower bound on $\tw(L(G))$ in terms of $\dd(G)^{\frac{3}{2}}$.

\paragraph{Upper Bounds.}
Now consider upper bounds on $\tw(L(G))$. Equivalent results by \citet{Bienstock,atserias} and \citet{calinescu} all show that \begin{align}\label{eq:atsbound}\tw(L(G)) \leq (\tw(G)+1)\Delta(G)-1.\end{align} To see this, consider a minimum width tree decomposition of $G$, and replace each bag $X$ by the set of edges incident with a vertex in $X$. This creates a tree decomposition of $L(G)$, where each bag contains at most $(\tw(G)+1)\Delta(G)$ edges. A similar argument can be used to prove that \begin{align}\label{eq:atsboundpw}\pw(L(G)) \leq (\pw(G)+1)\Delta(G)-1.\end{align} In Section~\ref{sec:upper}, we establish the following improvement.

\begin{theorem}
\label{theorem:maxdegintro}
For every graph $G$,
\begin{align*}\tw(L(G)) &\leq \tfrac{2}{3}(\tw(G)+1)\Delta(G) + \tfrac{1}{3}\tw(G)^2 + \tfrac{1}{3}\Delta(G) -1 \text{, and}\\
\pw(L(G)) &\leq \tfrac{1}{2}(\pw(G)+1)\Delta(G) + \tfrac{1}{2}\pw(G)^2 + \tfrac{1}{2}\Delta(G) -1.\end{align*}
\end{theorem}

Theorem~\ref{theorem:maxdegintro} is of primary interest when $\Delta(G) \gg \tw(G)$ or $\Delta(G) \gg \pw(G)$, in which case the upper bounds are $(\frac{2}{3}+o(1))\Delta(G)\tw(G)$ and $(\frac{1}{2}+o(1))\Delta(G)\pw(G)$. When $\Delta(G) < \tw(G)$ or $\Delta(G) < \pw(G)$, the bounds in \eqref{eq:atsbound} and \eqref{eq:atsboundpw} are better than those in Theorem~\ref{theorem:maxdegintro}.

In Section~\ref{sec:bipartite}, we show that this upper bound on $\pw(L(G))$ is sharp ignoring lower order terms. The key example here is $G=K_{p,q}$, which is of independent interest. Since $\tw(K_{p,q})=\pw(K_{p,q})=q$ and $\Delta(K_{p,q}) = p$ for $p \geq q$, Theorem~\ref{theorem:maxdegintro} implies that $\pw(L(K_{p,q})) \leq (\tfrac{1}{2} + o(1))pq$. Hence the following theorem is sufficient.
\begin{theorem}
\label{theorem:introbipart}
For all $p \geq q \geq 1$, $$\tfrac{1}{2}pq - 1\leq \tw(L(K_{p,q})) \leq \pw(L(K_{p,q})).$$
\end{theorem}
Theorem~\ref{theorem:introbipart} extends a previous result of \citet{lucenabramble}, who determined $\tw(L(K_{n,n})$ exactly, and a previous result from the PhD thesis of the first author \citep{mythesis}, which determined upper and lower bounds on the treewidth of line graphs of complete multipartite graphs. The bounds in \citep{mythesis} are equal when the graphs are regular, and are close when the graphs are almost regular. However, they say nothing when $p \gg q$, which is handled by Theorem~\ref{theorem:introbipart}.
%In particular, in Theorem~\ref{theorem:sepcompbi} we prove that for all $p \geq q$, the graph $K_{p,q}$ has $\tw(K_{p,q})=\pw(K_{p,q})=q$, $\Delta(K_{p,q}) = p$ and $\pw(L(K_{p,q})) \geq \tw(L(G)) \geq \frac{1}{2}pq$, and hence the pathwidth result of Theorem~\ref{theorem:maxdegintro} is sharp ignoring lower order terms.

%\begin{theorem}
%\label{theorem:cng}
%For all graphs $G$ with no isolated vertices, $$\cng'(G) = \tw(L(G))+1.$$ 
%\end{theorem}
%
%\begin{theorem}
%\label{theorem:cw}
%For all graphs $G$ with no isolated vertices, $$\cw(G) = \pw(L(G))+1.$$ 
%\end{theorem}

\section{Treewidth and the Congestion of Embeddings}
\label{sec:embed}

For a graph $G$, a \emph{tree decomposition} $(T,\mathcal{X})$ of $G$ is a tree $T$, together with $\mathcal{X}$, a collection of sets of vertices (called \emph{bags}) indexed by the nodes of $T$, such that: 
\begin{itemize*}
\item for all $v \in V(G)$, $v$ appears in at least one bag,
\item for all $v \in V(G)$, the nodes indexing the bags containing $v$ form a connected subtree of $T$, and 
\item for all $vw \in E(G)$, there is a bag containing both $v$ and $w$.
\end{itemize*}
(Often, we conflate a node and the bag indexed by that node, and refer to two bags being adjacent when their indexing nodes are adjacent and so on, for simplicity.)
The \emph{width} of a tree decomposition is the size of the largest bag, minus 1. The treewidth of $G$, denoted $\tw(G)$, is the minimum width over all tree decompositions of $G$. 

%Treewidth was initially defined by \citet{Halin76}. Treewidth is a core part of the famous Graph Minor Theorem of \citet{RS-GraphMinors}. Treewidth also has algorithmic applications; certain NP-Hard problems can be solved for graphs with bounded treewidth in polynomial time \citep{Bodlaender-AC93}.
%
A path decomposition is a tree decompositions where the underlying tree is a path. Pathwidth $\pw$ is defined analogously to treewidth but with respect to path decompositions. %Pathwidth is due to \citet{RS-GraphMinors-I}.

Given a tree decomposition of $L(G)$ with underlying tree $T$, for each edge $vw$ of $G$, let $S_{vw}$ denote the subtree of $T$ induced by the bags containing $vw$. (Recall each bag contains vertices of $L(G)$, which are edges of $G$.) %Also denote a bag which contains all edges incident to some $v \in V(G)$ as a \emph{base node} of $v$.

\begin{lemma}
\label{lemma:basenodes}
For every graph $G$ there exists a minimum width tree decomposition $(T,\mathcal{X})$ of $L(G)$ together with an assignment $\bb:V(G) \rightarrow V(T)$ such that for each edge $vw \in E(G)$, $S_{vw}$ is exactly the path in $T$ between $\bb(v)$ and $\bb(w)$. 
\end{lemma} %Note: What does |P| mean?
\begin{proof}
Let $(T,\mathcal{X})$ be a minimum width tree decomposition of $L(G)$ such that $\sum_{vw \in E(G)} |V(S_{vw})|$ is minimised. For each vertex $v$ of $G$, the edges incident to $v$ form a clique in $L(G)$ and thus, by the Helly property, there exists a bag of $T$ containing all edges incident to $v$. Hence for each $v$ choose one such node and declare it $\bb(v)$.

Consider an edge $vw \in E(G)$. Denote the path between $\bb(v)$ and $\bb(w)$ by $P_{vw}$. Since $vw$ is in the bags at $\bb(v)$ and $\bb(w)$, it follows $P_{vw} \subseteq S_{vw}$. If $|V(P_{vw})| < |V(S_{vw})|$ then we could obtain another tree decomposition of $L(G)$ by removing $vw$ from the bags of $V(S_{vw})-V(P_{vw})$, since each edge incident to $vw$ appears in $\bb(v) \cup \bb(w)$. However, such a tree decomposition would contradict our choice of $(T,\mathcal{X})$. Hence $P_{vw} = S_{vw}$, as required.
\end{proof}

%\begin{theorem}
%\label{theorem:basenodes}
%Let $(T,\mathcal{X})$ be a tree decomposition of $L(G)$ with all $|S_{vw}|$ minimised. Then each vertex $v \in V(G)$ has exactly one base node and $S_{vw}$ is a path with one endpoint the base node of $v$ and the other endpoint the base node of $w$.
%\end{theorem}
%\begin{proof}
%Obviously every $v \in V(G)$ has at least one base node by the Helly property given that the edges incident to $v$ form a clique in $L(G)$. 
%If for every $vw \in E(G)$, the subtree $S_{vw}$ is a path from a base node of $v$ to a base node of $w$ with no other base nodes of $v$ or $w$ on the path, then the theorem holds completely. (That is, this would also prove each vertex has exactly one base node.)
%Thus, suppose for the sake of a contradiction that some $vw$ violates this property. Choose a base node for $v$ and a base node for $w$ so that this distance between these nodes is minimised. Denote by $P_{vw}$ the path between these nodes. By the properties of the tree decomposition, $vw$ is in both these bags and all bags on the path between them, so $P_{vw} \subseteq S_{vw}$. If $S_{vw}=P_{vw}$ then $vw$ does not violate the aforementioned property. Hence $|P_{vw}| < |S_{vw}|$. However, since $|P_{vw}| \geq 1$ and since each edge incident to $vw$ appears in either a base node of $v$ or a base node of $w$, it would be sufficient to restrict $S_{vw}$ to $P_{vw}$ and still have a valid tree decomposition, contradicting the minimisation of $|S_{vw}|$. Thus there is a contradiction in either case.
%\end{proof}

We call $\bb(v)$ the \emph{base node} of $v$. What Lemma~\ref{lemma:basenodes} shows is that, in some sense, the best way to construct a tree decomposition of $L(G)$ is to choose a tree $T$, assign a base node for each $v \in V(G)$, and then place each edge in exactly the bags between the base nodes assigned to its endpoints---any other tree decomposition ``contains" such a tree decomposition inside of it. %We shall denote a tree decomposition created by assigning base nodes in this way a \emph{base node tree decomposition}.  

We can obtain a slightly stronger result that will be useful when proving our major theorems. Given $(T,\mathcal{X})$ and $\bb$ as guaranteed by Lemma~\ref{lemma:basenodes}, we can also ensure that each base node is a leaf and that $\bb$ is a bijection between vertices of $G$ and leaves of $T$. This is done as follows. If $\bb(v)$ is not a leaf, then simply add a leaf adjacent to $\bb(v)$, and let $\bb(v)$ be this leaf instead. Such an operation does not change the width of the tree decomposition. If some leaf $x$ is the base node for several vertices of $G$, then add a leaf adjacent to $x$ for each vertex assigned to $x$. Finally, if $x$ is a leaf that is not a base node, then delete $x$; this maintains the desired properties since a leaf is never an internal node of a path.
%For any base node that is not a leaf, add a leaf to the current base node and declare it to be the new base node. If some leaf is the base node for $r>1$ vertices, then add $r$ leaves adjacent to it so that each is a base node for one of the $r$ vertices. These alterations clearly do not alter the width of the tree decomposition. Finally, any leaf which is not a base node is not on any path, since a leaf on a path must be an endpoint of that path. Hence, we can simply delete any such leaf without damaging the tree decomposition.

%Note that after performing the above alterations, it is no longer the case that only the base node of $v$ contains all of the edges incident to $v$. (Or alternatively, there can now be more than one base node for a given vertex.) To avoid any issues, in this case we declare that only leaves can be base nodes---any other bag which contains all edges incident to some vertex $v$ is no longer a base node. By doing this, each base node is a leaf and there is a $1-1$ correspondence between leaves and base nodes. Denote such a tree decomposition to be a \emph{base node-leaf tree decomposition}.

We can improve this further. Given a tree $T$, we can root it at a node and orient all edges away from the root (that is, from the parent, to the child). In such a tree, a leaf is a node with outdegree 0. Say a rooted tree is \emph{binary} if every non-leaf node has outdegree 2. (That means that every non-leaf node has degree 3 except the root which has degree 2.) 

Given a tree decomposition, it is possible to root it and then modify the underlying tree so that each node has outdegree at most $2$, by (repeatedly) splitting a node with outdegree $3$ or more and distributing the children evenly amongst the two new nodes, where both new bags contain exactly the edges of the original bag. %Say there is a node $x$ of outdegree $3$ or more, such that $y,z$ are two of these children. Then delete the edges $xy,xz$ and create a new node $w$ together with the edges $xw,wy,wz$. (All of these edges are directed from the first node to the second.) Place in the bag at $w$ all the edges in the bag at $x$. 
This maintains all the properties of the tree decomposition and does not increase the width. %However, the outdegree of $x$ has fallen by $1$ and our new node has outdegree $2$. 
If $\bb$ is a mapping into the leaves, then this property is maintained by the splitting. In fact, in such a case, we can go further to obtain a binary tree; if $x$ is a non-root node with outdegree 1 then delete $x$ and an edge from its parent to its child, and if $x$ is a root with outdegree $1$ then delete $x$ and declare its child to be the new root. All of these results give the following key theorem.

\begin{theorem}
\label{theorem:goodtd}
For every graph $G$ there exists a minimum width tree decomposition $(T,\mathcal{X})$ of $L(G)$ together with an assignment $\bb:V(G) \rightarrow V(T)$ such that:
\begin{itemize*}
\item $T$ is a binary tree,
\item $\bb$ is a injection onto the leaves of $T$,
\item for each $vw \in E(G)$, $S_{vw}$ is exactly the path from $\bb(v)$ to $\bb(w)$.
\end{itemize*}
\end{theorem} 

Theorem~\ref{theorem:goodtd} has all the properties we require in order to prove our main results. It also leads to the following lower bound on $\tw(L(G))$ that is slightly stronger than \eqref{eq:easylb}.

\begin{proposition}
$\tw(L(G)) \geq \tw(G)-1.$
\end{proposition}
\begin{proof}
Let $k = \tw(L(G))+1$, and let $(T,\mathcal{X})$ be a tree decomposition of $L(G)$ of width $k-1$, together with an assignment $\bb$ as ensured by Lemma~\ref{lemma:basenodes}. Partially construct a tree decomposition of $G$ as follows: for each edge $vw \in E(G)$, arbitrarily choose one endpoint (say $v$) and place $v$ in all bags of $S_{vw}$ except $\bb(w)$, in which we place $w$. The size of a bag is at most $k$ since each edge contributes only one endpoint to a given bag. This is a tree decomposition of $G$, except if $vw \in E(G)$ then it is possible that $v$ and $w$ do not share a bag, but do appear in adjacent bags. For each such edge $vw \in E(G)$, call the edge $XY \in E(T)$ with $v \in X-Y$ and $w \in Y-X$ the edge \emph{corresponding} to $vw$. If $XY$ is the edge corresponding to both $vw,uz \in E(G)$, then subdivide it to create a new bag $X'=(X - \{v\}) \cup \{w\}$. Now $XX'$ corresponds to $vw$, and nothing else, and $X'Y$ corresponds to $uz$. Repeat this process so that every edge in $T$ corresponds to at most one edge of $G$. Finally, arbitrarily root $T$, and if $XY$ is the edge corresponding to $vw$ such that $Y$ is the child of $X$, then add $v$ to $Y$. Note that this increases the size of each bag by at most 1, and creates a tree decomposition for $G$. Thus $\tw(G) \leq k = \tw(L(G))+1$, as required.
%It is possible to modify the tree decomposition to ensure that $v,w$ share a common bag while only increasing the width by at most $1$. (We omit the precise proof of this.) Hence $\tw(G) \leq k$, as required.
\end{proof}

Theorem~\ref{theorem:goodtd} also shows a connection between $\tw(L(G))$ and embeddings of $G$ into a tree.
Consider the following definition by \citet{Bienstock}. Define an \emph{embedding} as an injective map from $V(G)$ into the leaves of a sub-cubic tree $T$. If $\pi$ is such an embedding and $vw \in E(G)$ then let $P_{vw}$ be the path from $\pi(v)$ to $\pi(w)$. The vertex congestion of $\pi$ is $$\max_{u \in V(T)}|\{vw \in E(G) : u \in V(P_{vw})\}|.$$ The \emph{vertex congestion of $G$}, denoted $\con(G)$, is the minimum congestion over all sub-cubic trees $T$ and choices of $\pi$. (\citet{Bienstock} also considered the \emph{edge congestion} of $G$ which counts the maximum number of paths $P_{vw}$ using an edge $e \in E(T)$. \citeauthor{Bienstock} showed that vertex and edge congestion are within a factor of $\frac{3}{2}$ of each other.) Graph embeddings into paths (which we discuss below) and infinite grids (for example \citep{Bhatt}) were studied prior to \citeauthor{Bienstock}. Embeddings have also been considered for hypercubes, see \citep{hypercube} for example. Determining $\con(G)$ is NP-hard \citep{Saks}. 

Observe that embeddings into sub-cubic trees are similar to our construction of tree decompositions in Theorem~\ref{theorem:goodtd}, and lead to the following theorem.

%In Theorem~\ref{theorem:goodtd}, we assume that $T$ is a binary tree. Similarly, it is possible to assume that $T$ is an unrooted binary tree. Hence, the following result follows directly from Theorem~\ref{theorem:goodtd}, by setting $\bb = \pi$. %FIX ALL OF THIS!
%
\begin{theorem}
\label{theorem:cng}
For every graph $G$, $$\con(G) = \tw(L(G))+1.$$ 
\end{theorem}
\begin{proof}
An embedding into the leaves of a sub-cubic tree is equivalent to an assignment of base nodes into the leaves. An edge $vw$ contributes to the congestion at a vertex $u$ of $T$ under an embedding $\pi$ if and only if $vw$ is in the bag of $u$ when $\pi$ is treated as an assignment. Thus $\tw(L(G))+1 \leq \con(G)$. Equality holds by Theorem~\ref{theorem:goodtd} since every binary tree is sub-cubic.
\end{proof}
%\begin{theorem}
%\label{theorem:cw}
%For all graphs $G$ with no isolated vertices, $$\cw(G) = \pw(L(G))+1.$$ 
%\end{theorem} 

%Then define the (edge) congestion of $\pi$ $\cng(\pi) = \max_{e \in E(T)}|\{vw \in E(G) : e \in P_{vw}\}|$ and define the vertex congestion of $\pi$ $\cng'(\pi) = \max_{u \in V(T)}|\{vw \in E(G) : u \in P_{vw}\}|$. Then let the (edge) congestion of $G$ be defined as $\cng(G) = \min_{\pi} \cng(\pi)$ and similarly $\cng'(G) = \min_{\pi} \cng'(\pi)$. \citet{Bienstock} is mostly concerned with (edge) congestion, but shows that $\cng(G) \leq \cng'(G) \leq \frac{3}{2}\cng(G)$ for all $G$. However, it should now be clear that the vertex congestion of a graph $G$ is closely related to the treewidth of $L(G)$. Above we showed it was permissible to assume $T$ is a rooted tree with outdegree at most $2$, similarly it follows that it is permissible to assume $T$ is a binary tree. This proves Theorem~\ref{theorem:cng}.

A similar result to Theorem~\ref{theorem:goodtd} holds for path decompositions. Much like our results on trees, it is reasonably clear that we can always ensure that $\bb$ is a bijection between vertices of $G$ and nodes of $P$. This gives the following lemma.

\begin{lemma}
\label{lemma:goodpd}
For every line graph $L(G)$ there exists a minimum width path decomposition $(P,\mathcal{X})$ together with an assignment $\bb:V(G) \rightarrow V(P)$ such that:
\begin{itemize*}
\item $P$ is a $|V(G)|$-node path,
\item $\bb$ is a $1-1$ mapping onto $P$,
\item for each $vw \in E(G)$, $S_{vw}$ is exactly the path from $\bb(v)$ to $\bb(w)$.
\end{itemize*}
\end{lemma}

From Lemma~\ref{lemma:goodpd} and by a similar argument to Theorem~\ref{theorem:cng}, the following holds.

\begin{theorem}
\label{theorem:vcngpath}
For every graph $G$, let $P$ be a $|V(G)|$-vertex path and $\Pi$ be the set of all bijections $\pi:V(G) \rightarrow P$. Then
$$\min_{\pi \in \Pi}|\max_{u \in V(P)} |\{vw \in E(G):u \in V(P_{vw})\}|| = \pw(L(G))+1,$$
where $P_{vw}$ is the path from $\pi(v)$ to $\pi(w)$.
\end{theorem}

Theorem~\ref{theorem:vcngpath} considers the minimum \emph{vertex} congestion of a graph embedding into a path. As mentioned previously, we may also consider the minimum \emph{edge} congestion of a graph embedding into a path. This topic is well studied \citep{cut1,cut2,cut3,cut4}; it is usually referred to as cutwidth. Specifically, if $\pi$ is a linear ordering of $G$ (that is, a bijection from $V(G)$ to $\{1,\dots,|V(G)|\}$), then the \emph{cutwidth} of $\pi$ is defined as $$\max_{1\leq i \leq |V(G)|} |\{vw \in E(G) : \pi(v) \leq i, \pi(w) > i\}|,$$ and the cutwidth of $G$, denoted $\cw(G)$, is the minimum cutwidth over all choices of $\pi$. Determining the cutwidth of a graph is NP-complete \citep{fanica}. Previously \citet{Golovach} proved that if $\Delta(G) \geq 2$ then $$\pw(L(G)) - \floor{\tfrac{\Delta(G)}{2}} + 1 \leq \cw(G) \leq \pw(L(G)).$$
(Note the result of Golovach concerns \emph{vertex separation number}, but this is equal to  pathwidth \citep{kinn}.) This result is the edge congestion equivalent to Theorem~\ref{theorem:vcngpath}. The lower bound here is sharp due to the star \citep{Golovach}. Thus there is a relationship between $\pw(L(G))$ and both the minimum vertex and minimum edge congestion of an embedding of $G$ into a path.%Hence we can see that the relationship between pathwidth and the vertex congestion is currently stronger than the equivalent relationship to the more studied edge congestion. %ALL OF THIS GOLOVACH NEEDS TO BE CHECKED AND DISCUSSED.

\section{Lower Bound in Terms of Average Degree}
\label{section:avgdeg}

This section proves Theorem~\ref{theorem:avgdegintro}. Say a graph $G$ is \emph{minimal} if $\dd(G-S) < \dd(G)$ for all non-empty $S \subsetneq V(G)$. For example, every connected regular graph is minimal. Given a set $X \subseteq V(G)$, let $\ee(X)$ denote the set of edges with both endpoints in $X$. Given $X,Y \subseteq V(G)$ such that $X \cap Y = \emptyset$, let $\ee(X,Y)$ denote the set of edges with one endpoint in each of $X$ and $Y$.

\begin{lemma}
\label{lemma:goodSfact}
If $G$ is a minimal graph and $S$ is a non-empty proper subset of $V(G)$, then
$$\frac{1}{2}\dd(G) < \frac{1}{|S|}\left(\left(\sum_{v \in S} \deg(v)\right) - |\ee(S)|\right).$$ 
\end{lemma}
\begin{proof}
Let $G' := G-S$, and note that $\dd(G') < \dd(G)$. Let $m:=|E(G)|$ and $n:=|V(G)|$. So,
$$\frac{2m}{n} = \dd(G) > \dd(G') = \frac{2(m - |\ee(S,V(G)-S)| - |\ee(S)|)}{n-|S|}.$$
Hence, $(m - |\ee(S,V(G)-S)| - |\ee(S)|)n < m(n-|S|)$ and $- |\ee(S,V(G)-S)|n - |\ee(S)|n < -m|S|$. Thus \begin{equation*}\frac{1}{2}\dd(G) = \frac{m}{n} < \frac{1}{|S|}\left(|\ee(S,V(G)-S)| + |\ee(S)|\right) = \frac{1}{|S|}\left(\left(\sum_{v \in S} \deg(v)\right) - |\ee(S)|\right). \qedhere\end{equation*}
\end{proof}

Theorem~\ref{theorem:avgdegintro} follows from the following lemma since every graph $G$ contains a minimal subgraph $H$ with $\dd(H) \geq \dd(G)$, in which case $L(H) \subseteq L(G)$ and $\tw(L(G)) \geq \tw(L(H))$.

\begin{lemma}
\label{lemma:avgdeg}
For every minimal graph $G$ with average degree $\dd(G)$, $$\tw(L(G)) > \frac{1}{8}\dd(G)^2 + \frac{3}{4}\dd(G) - 2.$$
\end{lemma}
\begin{proof}
If $\dd(G)=0$, then the lemma holds trivially. If $0<\dd(G)<2$, then $\tw(L(G)) \geq 0 = \frac{1}{2} + \frac{3}{2} -2 = \frac{1}{8}2^2 + \frac{3}{4}2 - 2 > \frac{1}{8}\dd(G)^2 + \frac{3}{4}\dd(G) - 2$, as required. Now assume that $\dd(G) \geq 2$. %In particular, $G$ has no isolated vertices.

Let $(T,\mathcal{X})$ be a tree decomposition for $L(G)$ as guaranteed by Theorem~\ref{theorem:goodtd}.
For each node $u$ of $T$, let $T_{u}$ denote the subtree of $T$ rooted at $u$ containing exactly $u$ and the descendants of $u$. Let $\zz(T_u)$ be the set of vertices of $G$ with base nodes in $T_u$. (Recall all base nodes are leaves.)
Call a node $u$ of $T$ \emph{significant} if $|\zz(T_u)| > \frac{1}{2}\dd(G)$ but $|\zz(T_v)| \leq \frac{1}{2}\dd(G)$ for each child $v$ of $u$. 

\begin{claim}
\label{claim:avgsig}
There exists a non-root, non-leaf significant node $u$.
\end{claim}
\begin{proof}
Starting at the root of $T$, begin traversing down the tree by the following rule: if some child $v$ of the current node has $|\zz(T_v)| > \frac{1}{2}\dd(G)$, then traverse to $v$, otherwise halt. Clearly this algorithm halts. 

For a leaf $v$, $|\zz(T_v)| = 1$. We only traverse to $v$ if $|\zz(T_v)| > \frac{1}{2}\dd(G) \geq \frac{1}{2}2=1$. Hence the algorithm halts at a non-leaf. 

Say the algorithm halts at the root. If $v,w$ are children of the root then $|\zz(T_{v})|,|\zz(T_{w})| \leq \frac{1}{2}\dd(G)$. Thus $|\zz(T_u)| = |\zz(T_{v})|+|\zz(T_{w})| \leq \dd(G) < |V(G)|$. But every base node is in $\zz(T_u)$. Hence the algorithm does not halt at the root.

Let $u$ be the node where the algorithm halts. It is not the root or a leaf. First, $|\zz(T_u)| > \frac{1}{2}\dd(G)$ given that we traversed to $u$. Second, if $v$ is a child of $u$, then $|\zz(T_v)| \leq \frac{1}{2}\dd(G)$. This shows that $u$ is a significant, as required.
\end{proof}

If $a,b$ are the children of $u$, let $A:=\zz(T_{a})$ and $B:=\zz(T_{b})$. Hence $|A|,|B| \leq \frac{1}{2}\dd(G)$ but $|A \cup B| > \frac{1}{2}\dd(G)$. Also $A \cap B = \emptyset$. Define
$$g(A,B) := \left(\sum_{v \in A} \deg(v)\right) + \left(\sum_{v \in B} \deg(v)\right) - |\ee(A)| - |\ee(B)| - |\ee(A,B)|.$$
\begin{claim}
$g(A,B) > \frac{1}{2}(|A|+|B|)\dd(G)$.
\end{claim}
\begin{proof}
Given that $|A \cup B| > \frac{1}{2}\dd(G) \geq \frac{1}{2}2$, it follows that $A \cup B \neq \emptyset$. Also, since $u$ is not the root and $\zz(T_u) = A \cup B$, it follows that $A \cup B \subsetneq V(G)$. Hence we may apply Lemma~\ref{lemma:goodSfact} to $A \cup B$. 
Hence
$$\frac{1}{2}\dd(G) < \frac{1}{|A \cup B|}\left(\left(\sum_{v \in A \cup B} \deg(v)\right) - |\ee(A \cup B)|\right).$$ 
%Note the following facts:
%\begin{align*}
%|A \cup B| &= |A| + |B| \\
%\left(\sum_{v \in A \cup B} \deg(v)\right) &= \left(\sum_{v \in A} \deg(v)\right) + \left(\sum_{v \in B} \deg(v)\right) \\
%|\ee(A \cup B)| &= |\ee(A)| + |\ee(B)| + |\ee(A,B)|
%\end{align*}
%Substituting these facts into our equation gives
By substitution,
\begin{equation*}\frac{1}{2}(|A|+|B|)\dd(G) < \left(\sum_{v \in A} \deg(v)\right) + \left(\sum_{v \in B} \deg(v)\right) - |\ee(A)| - |\ee(B)| - |\ee(A,B)| = g(A,B). \qedhere\end{equation*}
\end{proof}

Let $X$ be the bag indexed by $u$. The bag $X$ consists of every edge with exactly one endpoint in $A$ and every edge with exactly one endpoint in $B$. Thus,
%Now,
%\begin{align*}
%|\ee(A,V(G)-A)| &= \left(\sum_{v \in A} \deg(v)\right) - 2|\ee(A)| \\
%|\ee(B,V(G)-B)| &= \left(\sum_{v \in B} \deg(v)\right) - 2|\ee(B)| \\
%|\ee(A)| &\leq \binom{|A|}{2} = \frac{1}{2}|A|(|A|-1) \\ 
%|\ee(B)| &\leq \binom{|B|}{2} = \frac{1}{2}|B|(|B|-1)
%\end{align*}
%So,
\begin{align} \label{eq:lbX}
|X| &= |\ee(A,V(G)-A)| + |\ee(B,V(G)-B)| - |\ee(A,B)| \nonumber\\
&= \left(\sum_{v \in A} \deg(v)\right) - 2|\ee(A)| + \left(\sum_{v \in B} \deg(v)\right) - 2|\ee(B)| - |\ee(A,B)| \nonumber\\
&= g(A,B) - |\ee(A)| - |\ee(B)| \nonumber\\
&\geq g(A,B) - \tfrac{1}{2}|A|(|A|-1) - \tfrac{1}{2}|B|(|B|-1) \nonumber\\
&> \tfrac{1}{2}(|A|+|B|)\dd(G) - \tfrac{1}{2}|A|(|A|-1) - \tfrac{1}{2}|B|(|B|-1).
\end{align}

Define $\alpha,\beta$ such that $|A|=\alpha\dd(G)$ and $|B|=\beta\dd(G)$, and define $s:=\frac{1}{\dd(G)}$. Recall $|A|,|B| \leq \frac{1}{2}\dd(G)$ and $|A|+|B| > \frac{1}{2}\dd(G)$. Hence $|A|,|B| > 0$ and so $|A|,|B| \geq 1$. Thus $s \leq \alpha,\beta \leq \frac{1}{2}$ and $\alpha+\beta > \frac{1}{2}$. Substituting $|A|=\alpha\dd(G)$ and $|B|=\beta\dd(G)$ into \eqref{eq:lbX} gives
\begin{align*}
|X| &> \tfrac{1}{2}(\alpha\dd(G)+\beta\dd(G))\dd(G) - \tfrac{1}{2}\alpha\dd(G)(\alpha\dd(G)-1) - \tfrac{1}{2}\beta\dd(G)(\beta\dd(G)-1) \\
&= \tfrac{1}{2}\dd(G)^2(\alpha + \beta - \alpha^2 - \beta^2) + \tfrac{1}{2}\dd(G)(\alpha+\beta) \\
&= \tfrac{1}{2}\dd(G)^2(\alpha + \beta - \alpha^2 - \beta^2 + \alpha s +\beta s) \\
&= \tfrac{1}{2}\dd(G)^2((1+s)\alpha + (1+s)\beta - \alpha^2 - \beta^2).
\end{align*}
In Appendix~\ref{section:fmin} we prove that $(1+s)\alpha + (1+s)\beta - \alpha^2 - \beta^2 \geq \frac{1}{4} + \frac{3}{2}s - 2s^2$. Hence \begin{equation*}\tw(L(G))+1 \geq |X| > \tfrac{1}{2}\dd(G)^2(\tfrac{1}{4} + \tfrac{3}{2}s - 2s^2) = \tfrac{1}{8}\dd(G)^2 + \tfrac{3}{4}\dd(G) - 1. \qedhere
\end{equation*}
\end{proof}

Consider the case when $G=P_{n}^{k}$, the $k^{th}$-power of an $n$-vertex path. As $n \rightarrow \infty$, $\dd(G) = 2k - \gamma$ where $\gamma \rightarrow 0$. So Theorem~\ref{theorem:avgdegintro} states that $\tw(L(G)) > \frac{1}{2}k^2 + \frac{3}{2}k - 2 - \gamma(\frac{1}{2}k + \frac{3}{4} - \frac{1}{8}\gamma)$. Since $\frac{1}{2}k^2 + \frac{3}{2}k - 2$ is an integer, $\tw(L(G)) \geq \frac{1}{2}k^2 + \frac{3}{2}k - 2$. 
For an upper bound take a path decomposition of $L(G)$ in the form suggested by Lemma~\ref{lemma:goodpd}, ordering the base nodes in the same order as in the path in $G$. The largest bag contains $(\sum_{i=1}^{k-1} i) + 2k = \frac{1}{2}(k^2 - k) + 2k = \frac{1}{2}k^2 + \frac{3}{2}k$. Hence $\pw(L(P_{n}^{k})) \leq \frac{1}{2}k^2 + \frac{3}{2}k-1,$ and thus Theorem~\ref{theorem:avgdegintro} is almost precisely sharp for both treewidth and pathwidth---it is out by only $1$. %Note: do I say anything here about how we can almost use the ceiling to get exact sharpness? Don't know.

%(I think I can get a sharp result if I account for $\theta$ and $\epsilon$, but I don't like doing that given that $\theta$ depends on $|A|$ and $|B|$.)

\section{Lower Bound in Terms of Minimum Degree}
\label{section:mindeg}

We use similar techniques to those in Section~\ref{section:avgdeg} to prove a lower bound on $\tw(L(G))$ in terms of $\delta(G)$ instead of $\dd(G)$. This bound is superior when $G$ is regular or close to regular. Because this proof is so similar to that of Lemma~\ref{lemma:avgdeg}, we omit some of the details. However, we also take particular care with lower order terms, so that this result is sharp.

%\begin{theorem}
%\label{theorem:mindeg}
%$$\tw(L(G)) \geq 
%\begin{cases}
%\frac{1}{4}\delta(G)^2+\delta(G) -1 &\text{ when $\delta(G)$ is even}\\
%\frac{1}{4}\delta(G)^2+\delta(G)-\frac{5}{4} &\text{ when $\delta(G)$ is odd}.
%\end{cases}$$ 
%\end{theorem}
\begin{proof}[Proof of Theorem~\ref{theorem:mindegintro}.]
If $\delta(G) < 2$, then the result is trivial, since $\tw(L(G)) \geq 0$ whenever $L(G)$ contains at least one vertex. Now assume that $\delta(G) \geq 2$.

Let $(T,\mathcal{X})$ be a tree decomposition for $L(G)$ as guaranteed by Theorem~\ref{theorem:goodtd}.
For each node $u$ of $T$, let $T_{u}$ denote the subtree of $T$ rooted at $u$ containing exactly $u$ and the descendants of $u$. For any $T_u$, let $\zz(T_u)$ be the set of vertices of $G$ with base nodes in $T_u$.

Call a node $u$ of $T$ \emph{significant} if $|\zz(T_u)| > \frac{1}{2}\delta(G)$ but $|\zz(T_v)| \leq \frac{1}{2}\delta(G)$ for each child $v$ of $u$.
There exists a non-root, non-leaf significant node $u$.
This result follows by a argument similar to Claim~\ref{claim:avgsig}; run a similar traversal but only traverse down an edge when $|\zz(T_u)| > \frac{1}{2}\delta(G)$.
Let $a,b$ be the children of $u$, and define $A := \zz(T_{a})$ and $B:=\zz(T_{b})$. Hence $|A|,|B| \leq \frac{1}{2}\delta(G)$ and $|A|+|B| > \frac{1}{2}\delta(G)$. Since $|A|,|B|$ are integers, if $\delta(G)$ is odd then $|A|+|B| \geq \frac{1}{2}\delta(G)+\frac{1}{2}$, and if $\delta(G)$ is even then $|A|+|B| \geq \frac{1}{2}\delta(G)+1$. It also follows that $|A|,|B| \geq 1$. Define $\alpha,\beta,s$ such that $|A|=\alpha\delta(G)$, $|B| = \beta\delta(G)$ and $s = \frac{1}{\delta(G)}$. Thus
\begin{align*}
s \leq \alpha,&\beta \leq \frac{1}{2} \\
\alpha + \beta &\geq \begin{cases}
\frac{1}{2} + \frac{1}{2}s \text{ when $\delta(G)$ is odd}\\
\frac{1}{2} + s \text{ when $\delta(G)$ is even}
\end{cases}
\end{align*} 
Let $X$ be the bag indexed by $u$. Our goal is to show that $|X|$ is large. As in Lemma~\ref{lemma:avgdeg}, $$|X| = |\ee(A,V(G)-A)| + |\ee(B,V(G)-B)| - |\ee(A,B)|.$$
Note the following:
\begin{align*}
|\ee(A,V(G)-A)| &\geq \left(\sum_{v \in A} \deg(v) - |A|+1\right) \geq |A|\delta(G) - |A|^2 + |A| = ((1+s)\alpha - \alpha^2)\delta(G)^2.
\end{align*}
A similar result holds for $|\ee(B,V(G)-B)|$, and $|\ee(A,B)| \leq |A||B| = \alpha\beta\delta(G)^2$. Hence $$|X| \geq ((1+s)\alpha - \alpha^2 + (1+s)\beta - \beta^2 - \alpha\beta)\delta(G)^2.$$
In Appendix~\ref{section:fprimemin} we prove that
$$(1+s)\alpha - \alpha^2 + (1+s)\beta - \beta^2 - \alpha\beta \geq
\begin{cases}
\frac{1}{4}+s &\text{ when $\delta(G)$ is even}\\
\frac{1}{4}+s-\frac{1}{4}s^2 &\text{ when $\delta(G)$ is odd}.
\end{cases}$$
Thus 
\begin{equation*}\tw(L(G))+1 \geq |X| \geq \begin{cases}
\frac{1}{4}\delta(G)^2+\delta(G) &\text{ when $\delta(G)$ is even}\\
\frac{1}{4}\delta(G)^2+\delta(G)-\frac{1}{4} &\text{ when $\delta(G)$ is odd}.
\end{cases}\end{equation*}
\end{proof}

We now show that Theorem~\ref{theorem:mindegintro} is sharp. Let $C_{n}^{k}$ be the $k^{th}$-power of an $n$-vertex cycle $(1,\dots,n)$. Let the $i^{th}$ node in an $n$-vertex path be the base node for the $i^{th}$ vertex of $C_{n}^{k}$. It is easily seen each resulting bag has size at most $k^2 + 2k$. So $\pw(L(C_{n}^{k})) \leq k^2 + 2k -1 = \frac{1}{4}\delta(C_{n}^{k})^2 + \delta(C_{n}^{k}) - 1$, since $\delta(C_{n}^{k})=2k$. Hence Theorem~\ref{theorem:mindegintro} is precisely sharp when $\delta(G)$ is even. Now consider the odd case. Define the matching $X_1 := \{1(n-k+1),2(n-k+2),\dots,kn\}$, and if $n$ is even, also define the matching $X_2 := \{(k+1)(k+2), (k+3)(k+4), \dots, (n-k-1)(n-k)\}$. If $n$ is odd, let $H$ be the graph obtained from $C_{n}^{k}$ by deleting $X_1$; if $n$ is even instead delete $X_1 \cup X_2$. Then using the same base node assignment as above, it is easily seen that $$\pw(L(H)) \leq \begin{cases} &k^2 + k -1 \text{ if $n$ is odd,}\\ &k^2 + k -2 \text{ if $n$ is even.}\end{cases}$$ Since $\delta(H)=2k-1$, Theorem~\ref{theorem:mindegintro} is precisely sharp when $n$ is even and $\delta(G)$ is odd, and within `$+1$' when $n,\delta(G)$ are both odd.
%Observe that $L(P_{n}^{k}) = L(C_{n}^{k}) - X$, where $X := \{ab \in E(C_{n}^{k})| n-k+1 \leq a \leq n, 1 \leq b \leq k, n+b-a \leq k\}$. Since $|X|=\binom{k+1}{2}$, it follows from \eqref{eq:powerpath} that $$\pw(L(C_{n}^{k})) \leq \pw(L(P_{n}^{k})) + |X| \leq k^2 + 2k -1.$$ Given that $\delta(C_{n}^{k}) = 2k$, it follows Theorem~\ref{theorem:mindegintro} is precisely sharp in the even case. Now consider the odd case. If $n$ is even then removing a specific perfect matching from $C_{n}^{k}$ also gives a precisely sharp bound; if $n$ is odd then removing a specific near-perfect matching shows the lower bound is within one of being sharp.
Finally, applying Theorem~\ref{theorem:mindegintro} when $G=K_n$ agrees with the exact determination of $\pw(L(K_n))$ as given in \citep{mylinegraph-jgt,mythesis}, for both even and odd cases.

%Applying Theorem~\ref{theorem:mindegintro} when $G=K_n$ gives $$\pw(L(K_n)) \geq \tw(L(K_n)) \geq 
%\begin{cases}
%\frac{1}{4}(n-1)^2+n-2 &\text{ when $n$ is odd}\\
%\frac{1}{4}(n-1)^2+n-\frac{9}{4} &\text{ when $n$ is even}.
%\end{cases}$$ 
%This agrees with the exact determination of $\pw(L(K_n))$ as given in \citep{mylinegraph-jgt,mythesis}. Hence Theorem~\ref{theorem:mindegintro} is sharp.

\section{Upper Bounds}
\label{sec:upper}

%Consider the following na\"ive upper bound on $\tw(L(G))$.
%\begin{theorem}[\citep{Bienstock,atserias,calinescu}]
%\label{theorem:naiveub}
%For every graph $G$ with maximum degree $\Delta(G)$, $$\tw(L(G)) \leq (\tw(G)+1)\Delta(G)-1.$$
%\end{theorem} 
%\begin{proof}
%Consider a minimum width tree decomposition of $G$, and for each bag $X$ replace $X$ by the set of edges incident with a vertex in $X$. This creates a tree decomposition of $L(G)$, where each bag contains at most $(\tw(G)+1)\Delta(G)$ edges.
%\end{proof}
%
%A similar argument can be used to prove that $\pw(L(G)) \leq (\pw(G)+1)\Delta(G)-1$. Theorem~\ref{theorem:maxdegintro} improves these upper bounds when $\Delta(G) \geq \tw(G),\pw(G)$. Such graphs exist; if $p \geq q$ then $\tw(K_{p,q})=\pw(K_{p,q}) = q$ and $\Delta(K_{p,q}) = p$.
%
%\begin{proposition}
%For every $k$, every $d \geq k$ and every $n \geq d+1$, there exists an $n$-vertex graph with maximum degree $d$ and treewidth $k$.
%\end{proposition}
%\begin{proof}
%Let $G$ be the edge-maximal split graph with a clique of size $k$ and an independent set of size $n-k$. Thus $G$ has $n$ vertices, $\tw(G)=k$ as $G$ is a $k$-tree, and $\Delta(G) = n-1$. By deleting edges of $G$ it is possible to lower $\Delta(G)$ to $d$. If we delete edges while ensuring at least one $(k+1)$-clique remains in $G$, then the treewidth is maintained. This is possible as $d \geq k$.
%\end{proof}
%We now provide the proof of Theorem~\ref{theorem:maxdegintro}.
\begin{proof}[Proof of Theorem~\ref{theorem:maxdegintro}.]
Let $(T,\mathcal{X})$ be a tree decomposition of $G$ with width $k-1$ such that $T$ has maximum degree at most 3. By the discussion in Section~\ref{sec:intro}, we may assume that $\Delta(G) \geq k-1$. (The existence of such a $(T,\mathcal{X})$ is well known, and follows by a similar argument to Theorem~\ref{theorem:goodtd}.)

Say a vertex $v$ of $G$ is \emph{small} if $\deg(v) \leq k-1$ and \emph{large} otherwise. For each $v \in V(G)$, let $T_v$ denote the subtree of $T$ induced by the bags containing $v$. For each edge $e \in E(T)$, let $A(e),B(e)$ denote the two component subtrees of $T-e$. If $e$ is also an edge of $T_v$ for some $v$, then let $A(e,v),B(e,v)$ denote the two component subtrees of $T_v - e$, where $A(e,v) \subseteq A(e)$ and $B(e,v) \subseteq B(e)$. Let $\alpha(e,v)$ denote the set of neighbours of $v$ that appear in a bag of $A(e,v)$ and $\beta(e,v)$ denote the set of neighbours of $v$ that appear in a bag of $B(e,v)$.
Any vertex in both of these sets must be in the bags at both ends of $e$, but cannot be $v$ itself, and so $|\alpha(e,v) \cap \beta(e,v)| \leq k-1$.

\begin{claim}
\label{claim:agoodedge}
For every large $v \in V(G)$ there exists an edge $e \in T_v$ such that $|\alpha(e,v)|,|\beta(e,v)| \leq \frac{2}{3}\deg(v) + \frac{1}{3}(k-1)$. 
Moreover, if $T_v$ is a path, then there exists an edge $e \in T_v$ such that $|\alpha(e,v)|,|\beta(e,v)| \leq \frac{1}{2}\deg(v) + \frac{1}{2}(k-1)$.
\end{claim}
\begin{proof}
Assume for the sake of a contradiction that no such $e$ exists. Hence for all $e \in T_v$, either $|\alpha(e,v)|$ or $|\beta(e,v)|$ is too large. Direct the edge $e$ towards $A(e,v)$ or $B(e,v)$ respectively. (If both $|\alpha(e,v)|,|\beta(e,v)|$ are too large, then direct $e$ arbitrarily.) Given this orientation of $T_v$, there must be a sink, which we label $u$, and label the bag of $u$ by $X_u$.

Let $e_{1}, \dots, e_{d}$ be the edges of $T$ incident to $u$, where $d \in \{1,2,3\}$. Without loss of generality say that $e_i$ was directed towards $B(e_i,v)$ for all $e_i$. 

First, consider the case when $T_v$ is not a path. Hence $|\beta(e_i,v)| > \frac{2}{3}\deg(v) + \frac{1}{3}(k-1)$ for all $i$. 
If $d=3$, then $\sum_{i=1}^{3} |\beta(e_i,v)| > 2\deg(v) + (k-1)$. However, $\sum_{i=1}^{3} |\beta(e_i,v)|$ counts every neighbour of $v$ that is not in $X_u$ twice, since each subtree of $T_v - u$ is in $\beta(e_i,v)$ for two choices of $i$. It counts the neighbours of $v$ in $X_u$ three times, and there are at most $k-1$ of these (since $v \in X_u$). Thus $\sum_{i=1}^{3} |\beta(e_i,v)| \leq 2\deg(v) + (k-1)$, which is a contradiction.
If $d=2$, then $\sum_{i=1}^{2} |\beta(e_i,v)| > \frac{4}{3}\deg(v) + \frac{2}{3}(k-1)$. However, $\sum_{i=1}^{2} |\beta(e_i,v)|$ counts every neighbour of $v$ not in $X_u$ once, and every neighbour of $v$ in $X_u$ twice, so $\sum_{i=1}^{2} |\beta(e_i,v)| \leq \deg(v) + (k-1)$. But then $\deg(v) < k-1$, contradicting the fact that $v$ is large.
If $d=1$, then $|\beta(e_1,v)| > \frac{2}{3}\deg(v) + \frac{1}{3}(k-1)$. However, $\beta(e_1,v)$ is contained within $X-u$ and so $|\beta(e_1,v)| \leq k-1$, and again $\deg(v) < k-1$, a contradiction.

Now, consider the case when $T_v$ is a path. Hence $|\beta(e_1,v)|,|\beta(e_2,v)| > \frac{1}{2}\deg(v) + \frac{1}{2}(k-1)$.
If both $e_i$ exist then $\sum_{i=1}^{2} |\beta(e_i,v)| > \deg(v) + (k-1)$, but $\sum_{i=1}^{2} |\beta(e_i,v)|$ counts every neighbour of $v$ not in $X_u$ once, and the neighbours of $v$ in $X_u$ twice. Thus $\sum_{i=1}^{2} |\beta(e_i,v)| \leq \deg(v) + (k-1)$, a contradiction.
If $u$ has degree 1, then $|\beta(e_1,v)| > \frac{1}{2}\deg(v) + \frac{1}{2}(k-1)$ but $\beta(e_1,v)$ is contained within $X_u$, and so $\deg(v) < k-1$, a contradiction.
\end{proof}

For each small vertex $v$ of $G$, arbitrarily select a base node in $T_v$. For each large vertex $v$ of $G$, select an edge $e$ of $T_v$ as guaranteed by Claim~\ref{claim:agoodedge}. Subdivide $e$ and declare the new node to be $\bb(v)$, the base node of $v$. If $e$ is selected for several different vertices, then subdivide it multiple times and assign a different base node for each vertex of $G$ that selected $e$. Denote the tree $T$ after all of these subdivisions as $T'$. Together, this underlying tree $T'$ and the assignment $\bb$ gives a tree decomposition of $L(G)$ in the same form as Lemma~\ref{lemma:basenodes}. Label the set of bags for this tree decomposition by $\mathcal{X}'$, so the tree decomposition of $L(G)$ is $(T',\mathcal{X}')$. It remains to bound the width of this tree decomposition.

For each bag $X'$ of $\mathcal{X}'$, define a \emph{corresponding bag} in $\mathcal{X}$ as follows. If $X'$ is indexed by a node $x$ in $T'$ that is also in $T$, then the corresponding bag is simply the bag of $\mathcal{X}$ indexed by $x$ in $T$. If $X'$ is indexed by a subdivision node created by subdividing the edge $e$, then the corresponding bag is one of the bags of $\mathcal{X}$ indexed by the endpoints of $e$, chosen arbitrarily.

The following two claims give enough information to bound the width of $(T',\mathcal{X}')$.

\begin{claim}
\label{claim:correspond}
If $X'$ is a bag of $\mathcal{X}'$ with corresponding bag $X$, and $vw$ is an edge of $G$ in $X'$, then $v \in X$ or $w \in X$. 
\end{claim}
\begin{proof}
Assume for the sake of a contradiction that $vw \in X'$ but neither $v$ nor $w$ is in $X$. Hence $X \notin V(T_v) \cup V(T_w)$. Thus $T_v$ and $T_w$ are contained in $T-X$. If $T_v$ and $T_w$ are contained in different components of $T-X$, then $V(T_v) \cap V(T_w) = \emptyset$, but this is not possible given that $vw \in E(G)$. Thus $T_v$ and $T_w$ are contained in the same component of $T-X$. However, $\bb(v)$ and $\bb(w)$ are assigned inside of $T_v$ and $T_w$ respectively (perhaps after some edges are subdivided, but this does not alter their positions relative to $X$). Hence the path from $\bb(v)$ to $\bb(w)$ in $T'$ does not include $X'$, and so $vw \notin X'$. This is a contradiction.
\end{proof}

\begin{claim}
\label{claim:limitedge}
If $v$ is a large vertex and $X' \in \mathcal{X}'$ is not $\bb(v)$, then $X'$ contains at most $\frac{2}{3}\deg(v) + \frac{1}{3}(k-1)$ edges incident to $v$. Moreover, if $T_v$ is a path, then $X'$ contains at most $\frac{1}{2}\deg(v) + \frac{1}{2}(k-1)$ edges incident to $v$.
\end{claim}
\begin{proof}
As $v$ is large, $\bb(v)$ is a subdivision node, and $T'-\bb(v)$ has two components. Label these subtrees $L_v$ and $R_v$, and let $e \in E(T)$ be the edge that was subdivided to create $\bb(v)$. Let $w$ be a neighbour of $v$ such that $\bb(w)$ is in $L_v$. (Note each neighbour is in $L_v$ or $R_v$ as $\bb(v) \neq \bb(w)$.) Thus, without loss of generality, $w$ appears in a bag of $A(e,v)$, and so $w \in \alpha(e,v)$. By Claim~\ref{claim:agoodedge}, $|\alpha(e,v)| \leq \frac{2}{3}\deg(v) + \frac{1}{3}(k-1)$, and so at most $\frac{2}{3}\deg(v) + \frac{1}{3}(k-1)$ neighbours of $v$ may have their base node in $L_v$, and the same bound holds for $R_v$.

If $vw$ is an edge in $X'$, then $X'$ is on the unique path in $T'$ between $\bb(v)$ and $\bb(w)$, without being $\bb(v)$ itself. So $X'$ is, without loss of generality, in $L_v$, and then so is $\bb(w)$. Hence there are at most $\frac{2}{3}\deg(v) + \frac{1}{3}(k-1)$ such choices of $w$ and hence $\frac{2}{3}\deg(v) + \frac{1}{3}(k-1)$ edges in $X'$ incident to $v$.

If $T_v$ is a path, then the result follows from the alternate upper bound in Claim~\ref{claim:agoodedge}.
\end{proof}

We now determine an upper bound on the size of a bag $X' \in \mathcal{X'}$. We count the edges of $X'$ by considering the number of edges a given vertex $v$ of $G$ contributes to $X'$. By Claim~\ref{claim:correspond}, only the at most $k$ vertices of the corresponding bag $X$ contribute anything to $X'$.

\begin{itemize*} %NOTE USE ITEMIZE*
\item If $v$ is small, it contributes at most $\deg(v) \leq k-1$ edges to $X'$.
\item If $v$ is large and $X' \neq \bb(v)$, then by Claim~\ref{claim:limitedge}, $v$ contributes at most $\frac{2}{3}\Delta(G) + \frac{1}{3}(k-1)$ edges to $X'$. Given that $\Delta(G) \geq k-1$, this is at least $k-1$.
\item If $v$ is large and $X' = \bb(v)$, then $v$ contributes at most $\Delta(G)$ edges. This is at least $\frac{2}{3}\Delta(G) + \frac{1}{3}(k-1)$ as $\Delta(G) \geq k-1$. However, $X' = \bb(v)$ for at most one $v$.
\end{itemize*}

So in the worst case, there are $k$ vertices in the corresponding bag, all of which are large and contribute the maximum number of edges, which is $\frac{2}{3}\Delta(G) + \frac{1}{3}(k-1)$ for $k-1$ vertices and $\Delta(G)$ for one vertex.
Hence 
\begin{align*}
|X'| &\leq (k-1)(\tfrac{2}{3}\Delta(G) + \tfrac{1}{3}(k-1)) + \Delta(G)
= \tfrac{2}{3}k\Delta(G) + \tfrac{1}{3}(k-1)^2 + \tfrac{1}{3}\Delta(G).
\end{align*}

If we set $(T,\mathcal{X})$ to be a minimum width tree decomposition, then $k-1 = \tw(G)$, and so
$$\tw(L(G)) \leq \tfrac{2}{3}(\tw(G)+1)\Delta(G) + \tfrac{1}{3}\tw(G)^2 + \tfrac{1}{3}\Delta(G)-1.$$

Alternatively, if we let $(T,\mathcal{X})$ be a minimum width path decomposition, then $k-1 = \pw(G)$, and we can use the alternate upper bound in Claim~\ref{claim:limitedge} given that $T_v$ is always a path. Since $T'$ was created by subdividing edges, $T'$ is also a path. Hence

\begin{equation*}\pw(L(G)) \leq \tfrac{1}{2}(\pw(G)+1)\Delta(G) + \tfrac{1}{2}\pw(G)^2 + \tfrac{1}{2}\Delta(G)-1. \qedhere \end{equation*}
\end{proof} 

We now consider a few extensions of Theorem~\ref{theorem:maxdegintro}. For an outerplanar graph $G$, which has treewidth at most $2$, \eqref{eq:atsbound} proves that $\tw(L(G)) \leq 3\Delta(G)-1$. Theorem~\ref{theorem:maxdegintro} proves that $\tw(L(G)) \leq \frac{7}{3}\Delta(G) + \frac{1}{3}$. We can do better as follows.

\begin{corollary}
If $G$ is outerplanar, then $\tw(L(G)) \leq 2\Delta(G)+1.$
\end{corollary}
\begin{proof}[Proof Sketch]
In Theorem~\ref{theorem:maxdegintro}, if it were possible to select a tree decomposition such that $T_v$ was a path for each $v \in V(G)$, then it would be possible to achieve an upper bound of $\tw(L(G)) \leq \tfrac{1}{2}(\tw(G)+1)\Delta(G) + \tfrac{1}{2}\tw(G)^2 + \tfrac{1}{2}\Delta(G)-1.$ Since $G$ is outerplanar, let $G'$ be an outerplanar triangulation such that $G \subseteq G'$, and let $T$ be the weak dual of $G'$. Take $(T,(B_{x})_{x \in V(T)})$ as the tree decomposition of $G$, where the bag $B_x$ is the set of three vertices on the boundary of the face corresponding to $x \in V(T)$. Note that this tree decomposition has width $2$ and $T_v$ is a path for all $v \in V(G)$. Hence the result follows.
\end{proof}

It is plausible that Theorem~\ref{theorem:maxdegintro} can be further improved. %Firstly, it may be possible to remove the $\tw(G)^2$ and $\pw(G)^2$ terms if there is some way to restrict the number of neighbours which could either be in $\alpha(e,v)$ or $\beta(e,v)$. This is especially true of pathwidth. If this could be done, the improved Theorem~\ref{theorem:maxdegintro} would be an improvement over Theorem~\ref{theorem:naiveub} even when $\Delta(G) \leq \tw(G)$.
The following conjecture is the strongest possible in this direction. 

\begin{conjecture}
\label{conjecture:strong}
For every graph $G$ with maximum degree $\Delta(G)$, $$\tw(L(G)) \leq \tfrac{1}{2}(\tw(G)+1)\Delta(G) - 1.$$
\end{conjecture}
This conjecture seems very strong, and indeed it seems challenging even in the treewidth $2$ case. Nevertheless, we now prove it for trees, thus providing some supporting evidence.

\begin{proposition}
\label{prop:tree}
If $\tw(G)=1$ then $\tw(L(G)) = \Delta(G)-1$.
\end{proposition}
\begin{proof}
We may assume $G$ is a tree. Construct a tree decomposition for $L(G)$ by taking the underlying tree to be $G$ itself and letting $\bb(v)=v$. Then each bag contains exactly the edges of $G$ incident to the vertex, and so $\tw(L(G)) \leq \Delta(G)-1$. This is also a lower bound given that $L(G)$ contains a clique of order $\Delta(G)$.
\end{proof}
%
%We also present the following more achievable conjecture.
%\begin{conjecture}
%\label{conjecture:weak}
%For every graph $G$ with maximum degree $\Delta(G) \geq \tw(G)$,
%$$\tw(L(G)) \leq \tfrac{1}{2}(\tw(G)+1)\Delta(G) + \tfrac{1}{2}\tw(G)^2 + \tfrac{1}{2}\Delta(G) -1.$$
%\end{conjecture}
%
%Conjecture~\ref{conjecture:weak} is itself interesting since pathwidth result in Theorem~\ref{theorem:maxdegintro} is sharp up to lower order terms, due to $L(K_{p,q})$. If Conjecture~\ref{conjecture:weak} holds, then it would also be sharp up to lower order terms. 

\section{Treewidth of \texorpdfstring{\boldmath $L(K_{p,q})$ \unboldmath}{L(K(p,q))}}
\label{sec:bipartite}
%In the PhD thesis of the first author \citep{mythesis}, upper and lower bounds on the treewidth of line graphs of complete multipartite graphs are determined. These bounds are equal when the graphs are regular, and are close when the graphs are almost regular. (These results extend the previous determination of $\tw(L(K_{n,n}))$ by \citet{lucenabramble}). Below we provide a lower bound on $\tw(L(K_{p,q}))$ that is still useful when $p,q$ are not close to each other.

\begin{proof}[Proof of Theorem~\ref{theorem:introbipart}.] The graph $L(K_{p,q})$ is isomorphic to $K_{p} \square K_{q}$, the Cartesian product of $K_p$ and $K_q$, which can be thought of as a grid with $p$ rows and $q$ columns such that each row and column is a clique. A \emph{separator} of $G$ is a set of vertices $X$ such that $V(G-X)$ can be partitioned into at most three parts $A_1,A_2,A_3$ such that $|A_i| \leq \frac{1}{2}|V(G-X)|$ for all $i$, and no edge of $G-X$ has an endpoint in more than one part. (See \citep{mythesis,mypttt} for more explanation on separators.) A well-known result of \citet{RS-GraphMinorsII-JAlg86} states that every graph $G$ has a separator of order $\tw(G)+1$. Let $G = L(K_{p,q}) = K_p \square K_q$. It is sufficient to show that if $X$ is a separator of $G$ then $|X| \geq \tfrac{1}{2}pq$.

%
%\begin{theorem}[\citep{RS-GraphMinorsII-JAlg86}]
%\label{theorem:robsey}
%Every graph $G$ has a separator of order $\tw(G)+1$.
%\end{theorem}
%
%Recall if $p \geq q$, then $\Delta(K_{p,q}) = p$ and $\tw(K_{p,q}) = \pw(K_{p,q}) = q$. Thus Theorem~\ref{theorem:maxdegintro} gives that $\tw(L(K_{p,q})) \leq \pw(L(K_{p,q})) \leq \frac{1}{2}pq + p + \frac{1}{2}q^{2}-1$. Hence the following theorem shows that the pathwidth result in Theorem~\ref{theorem:maxdegintro} is sharp up to lower order terms.
%
%\begin{theorem}
%\label{theorem:sepcompbi}
%For all $p,q$ there is a graph $G$ with $\Delta(G) = p$ and $\tw(G) = q$ and $\tw(L(G)) \geq \frac{1}{2}\Delta(G)\tw(G)-1$.
%If $X$ is a separator of $G=K_{p} \square K_{q}$ then $|X| \geq \frac{1}{2}pq.$
%\end{theorem}
%\begin{proof}
Label the parts of $V(G-X)$ by $A_1,A_2,A_3$. Clearly $|A_1|+|A_2|+|A_3| + |X| = |V(K_{p} \square K_q)| = pq$. Consider a row $R$ of $G$. No two vertices of $R$ are in different parts, since $R$ forms a clique. Thus $R$ is a subset of $A_i \cup X$ for some $i$; colour $R$ by $i$. If no vertex of $R$ is in $G-X$, then colour $R$ arbitrarily. Colour columns similarly. Thus a vertex is in $A_i$ only if its row and column are both coloured $i$. (However, such vertices are not necessarily in $A_i$.) Define $x_i,y_i,z_i$ such that $x_ip$ is the number of rows coloured $i$, $y_iq$ is the number of columns coloured $i$, and $z_ipq$ is the number of vertices not in $A_i$ whose row and column is coloured $i$. Then $|A_i|=(x_iy_i - z_i)pq$. Define $\alpha_i := \frac{|A_i|}{pq}$. %then it suffices to prove that $$\alpha_1 + \alpha_2 + \alpha_3 \leq \tfrac{1}{2}.$$
%Choose $x_1,y_1,z_1,x_2,y_2,z_2, x_3,y_3,z_3$ to maximise $\alpha_1 + \alpha_2 + \alpha_3$ subject to the 
Clearly, these variables satisfy the following \emph{basic} constraints:
\begin{align*}
0 \leq x_i,y_i \,\,\forall i \qquad 0 \leq z_i \leq x_iy_i \,\,\forall i \qquad
x_1 + x_2 + x_3 =1 \qquad y_1 + y_2 + y_3=1,
\end{align*}
and the following \emph{balance} constraints (since $|A_i| \leq \frac{1}{2}(|A_1|+|A_2|+|A_3|)$):
\begin{align*}
\alpha_1 \leq \alpha_2 + \alpha_3 \qquad 
\alpha_2 \leq \alpha_3 + \alpha_1 \qquad
\alpha_3 \leq \alpha_1 + \alpha_2.
\end{align*}

In Appendix~\ref{sec:bipartworking} we prove that $\alpha_1 + \alpha_2 + \alpha_3 \leq \frac{1}{2}$, implying $|A_1|+|A_2|+|A_3| \leq \frac{1}{2}pq$ and $|X| \geq \frac{1}{2}pq$, as desired.
\end{proof}

\section{Alternate Lower Bounds}
\label{sec:altlow}

Given the format of Theorem~\ref{theorem:maxdegintro} and Conjecture~\ref{conjecture:strong}, we might hope for some analogous lower bound in terms of minimum degree and treewidth, or average degree and treewidth. In particular, does there exist some constant $c>0$ such that any of the following hold? 
\begin{align}
\tw(L(G)) &\geq c\,\tw(G)\delta(G) &\qquad \tw(L(G)) &\geq c\,\tw(G)\dd(G) \nonumber\\*
\pw(L(G)) &\geq c\,\pw(G)\delta(G) &\qquad \pw(L(G)) &\geq c\,\pw(G)\dd(G) \label{eq:strength}
\end{align}

Each of these inequalities would be qualitative strengthenings of our results in Sections~\ref{section:avgdeg} and \ref{section:mindeg}, since $\pw(G) \geq \tw(G) \geq \delta(G)$ and $\pw(G) \geq \tw(G) > \frac{1}{2}\dd(G)$. However, we now prove that none of these inequalities hold---thanks to Bruce Reed for this example. This implies that Theorems~\ref{theorem:avgdegintro} and \ref{theorem:mindegintro} are best possible in the sense that we cannot replace $\delta(G)$ or $\dd(G)$ by $\tw(G)$.

For positive integers $n,k$ construct the following graph $H_{n,k}$. Begin with the $n \times n$ grid, and for each vertex $v$ of the grid, have $k-\deg(v)$ disjoint cliques of order $k+1$. For each such clique $C$, add a single edge from a single vertex of $C$ to $v$. Every vertex of this graph has degree $k$, except those vertices of the cliques that are adjacent to vertices of the grid, which have degree $k+1$. Hence $\delta(H_{n,k}) = k$ and $\dd(H_{n,k}) > k$. Since $H_{n,k}$ contains an $n \times n$ grid, it follows that $\tw(H_{n,k}) \geq n$. We now prove a weak upper bound on $\tw(L(H_{n,k}))$.

\begin{lemma}
\label{lemma:pwhnk}
$\tw(L(H_{n,k})) \leq \pw(L(H_{n,k})) \leq 4n+O(k^3)$.
\end{lemma}
\begin{proof}
Let $v$ be a vertex of the grid in $H_{n,k}$, and let $A_v$ be the set containing the vertex $v$ together with all vertices of the cliques $C$ where there is an edge from $C$ to $v$. The sets $A_v$ form a partition of $V(H_{n,k})$. Let $P$ be an $n^2$-node path, and label the vertices of the grid $1,\dots,n^2$ considering rows from top to bottom, and going along each row from left to right. Then let the $i^{th}$ node of $P$ be the base node for all $w \in A_i$. This defines a path decomposition of $L(H_{n,k})$. Let $X_i$ be the bag indexed by the $i^{th}$ node. By the labelling, for each edge $ab$ of the grid, $|b-a| \leq n$. Hence if $ab \in X_i$ then without loss of generality, $i-n \leq a \leq i$. Thus there are $n+1$ possible choices of $a$, and each such $a$ may contribute at most $4$ such edges, and thus $X_i$ contains at most $4n+4$ such edges. Now consider edges without both endpoints in the grid. If $w \in A_j - \{j\}$, then every neighbour of $w$ is in $A_j$, and as such the edges with at least one endpoint in $A_j - \{j\}$ appear in $X_i$ only when $i=j$. Thus $|X_i| \leq 4n+4 + |\{e: e \text{ has at least one endpoint in } A_i-\{i\}\}| \leq 4n+O(k^3)$.
\end{proof}

Each possible strengthening in \eqref{eq:strength} would imply that $\tw(L(H_{n,k})) \geq cnk$ or $\pw(L(H_{n,k})) \geq cnk$ where $c$ is some constant, which contradicts Lemma~\ref{lemma:pwhnk} for $n \gg k \gg \tfrac{1}{c}$. Hence none of these strengthenings hold.

\appendix
\section{Appendix A}
\label{section:fmin}

%Note: I'm not sure this is completely okay given that the boundary is, in one case, a strict inequality rather than an equality. Note: Leave it as is for the moment, should check with David about this.

%(I can fix that if we put in $\epsilon$, however that's rather a lot more work, I'm afraid, and this is a little long as is. It might be unavoidable, however.)

Here we prove that for $s \leq \alpha,\beta \leq \frac{1}{2}$ and $\alpha+\beta > \frac{1}{2}$ and $0<s \leq \frac{1}{2}$, $$f(\alpha,\beta) := (1+s)\alpha + (1+s)\beta - \alpha^2 - \beta^2 \geq \tfrac{1}{4} + \tfrac{3}{2}s - 2s^2.$$
We do this using calculus of two variables. Any minimum point is either at a critical point, along the boundary of the defined region, or at a corner point. It is sufficient to show that $f(\alpha,\beta)$ evaluates to $\frac{1}{4} + \frac{3}{2}s - 2s^2$ at the minimum point.

For any critical point, the second partial derivative test shows that it is a local maximum:
\begin{equation*}
f_{\alpha\alpha}(\alpha,\beta) = -2 \qquad
f_{\beta\beta}(\alpha,\beta) = -2 \qquad
f_{\alpha\beta}(\alpha,\beta) = 0.
\end{equation*}
Hence
\begin{align*}
D(\alpha,\beta) &= f_{\alpha\alpha}(\alpha,\beta)f_{\beta\beta}(\alpha,\beta) - (f_{\alpha\beta}(\alpha,\beta))^2 = 4 > 0. 
\end{align*}
Since $f_{\alpha\alpha}(\alpha,\beta) < 0$, this shows any critical point is a local maximum.

Along the boundary of the region, we consider functions of one variable. However, along most of the boundary, either $\alpha$ or $\beta$ is constant (either $s$ or $\frac{1}{2}$), and in such cases our one variable functions are equivalent to either $f_{\alpha,\alpha}$ or $f_{\beta,\beta}$. By the second derivative test any critical point will be a local maximum.

Slightly more care is required along the boundary defined by $\alpha + \beta = \frac{1}{2}.$ %(I'm not sure this is entirely correct.)

An easy rearrangement gives $f(\alpha,\beta) = (1+s)(\alpha+\beta) - \alpha^2 - \beta^2$. Then 
$$
f(\alpha,\tfrac{1}{2} - \alpha) = (1+s)\tfrac{1}{2} - \alpha^2 - (\tfrac{1}{2} - \alpha)^2 = \tfrac{1}{4} + \tfrac{1}{2}s + \alpha - 2\alpha^2.
$$
Interpreting the above as a function in one variable, the second derivative test shows any critical point along the boundary is a local maximum.

All that remains is to consider the corner points; the smallest evaluation at a corner will be the minimum of $f(\alpha,\beta)$ in the given region. The corner points are $(\frac{1}{2},\frac{1}{2}),(\frac{1}{2},s),(\frac{1}{2}-s,s),(s,\frac{1}{2}-s)$ and $(s,\frac{1}{2})$. Given that $f(\alpha,\beta) = f(\beta,\alpha)$, it suffices to check the following three points.
\begin{align*}
f(\tfrac{1}{2},\tfrac{1}{2}) &= (1+s)\tfrac{1}{2} + (1+s)\tfrac{1}{2} - \tfrac{1}{4}  - \tfrac{1}{4} = 1 + s - \tfrac{1}{2} = \tfrac{1}{2}+s \\
f(\tfrac{1}{2},s) &= (1+s)\tfrac{1}{2} + (1+s)s - \tfrac{1}{4}  - s^2 = \tfrac{1}{4} + \tfrac{3}{2}s \\
f(\tfrac{1}{2} - s,s) &= (1+s)\tfrac{1}{2} - (\tfrac{1}{2}-s)^2 - s^2  = \tfrac{1}{2} + \tfrac{1}{2}s - \tfrac{1}{4} + s - s^2 - s^2 = \tfrac{1}{4} + \tfrac{3}{2}s - 2s^2
\end{align*}
If $\frac{1}{4} + \frac{3}{2}s > \frac{1}{2} + s$, then $s > \frac{1}{2}$. 
Given that $s \leq \frac{1}{2}$, it follows that $f(\frac{1}{2},\frac{1}{2}) \geq f(\frac{1}{2},s)$. As $s>0$, it follows $f(\frac{1}{2}-s,s) < f(\frac{1}{2},s)$. Hence $f(\alpha,\beta)$ is minimal at $(\frac{1}{2}-s,s)$, and so $f(\alpha,\beta) \geq \frac{1}{4} + \frac{3}{2}s - 2s^2.$

\section{Appendix B}
\label{section:fprimemin}

Here we prove that $$h(\alpha,\beta) := (1+s)\alpha - \alpha^2 + (1+s)\beta - \beta^2 - \alpha\beta \geq
\begin{cases}
\frac{1}{4}+s &\text{ when $\delta(G)$ is even}\\
\frac{1}{4}+s-\frac{1}{4}s^2 &\text{ when $\delta(G)$ is odd.}
\end{cases}$$
given that $0<s \leq \alpha,\beta \leq \frac{1}{2}$ and that $$\alpha + \beta \geq
\begin{cases}
\frac{1}{2} + s &\text{ when $\delta(G)$ is even}\\
\frac{1}{2} + \frac{1}{2}s &\text{ when $\delta(G)$ is odd}. 
\end{cases}$$

For any critical point, the second partial derivative test shows that it is a local maximum:
\begin{equation*}
h_{\alpha\alpha}(\alpha,\beta) = -2 \qquad
h_{\beta\beta}(\alpha,\beta) = -2 \qquad
h_{\alpha\beta}(\alpha,\beta) = -1.
\end{equation*}
Hence
\begin{align*}
D(\alpha,\beta) &= h_{\alpha\alpha}(\alpha,\beta)h_{\beta\beta}(\alpha,\beta) - (h_{\alpha\beta}(\alpha,\beta))^2 = 3 > 0. 
\end{align*}
Since $h_{\alpha\alpha}(\alpha,\beta) < 0$, this shows any critical point is a local maximum.

Along the boundary of the region, we consider functions of one variable. However, along most of the boundary, either $\alpha$ or $\beta$ is constant (either $s$ or $\frac{1}{2}$), and in such cases our one variable functions are equivalent to either $h_{\alpha,\alpha}$ or $h_{\beta,\beta}$. By the second derivative test any critical point will be a local maximum.

Slightly more care is required along the boundary defined by $$\alpha + \beta =
\begin{cases}
\frac{1}{2} + s &\text{ when $\delta(G)$ is even}\\
\frac{1}{2} + \frac{1}{2}s &\text{ when $\delta(G)$ is odd}.
\end{cases}$$

An easy rearrangement gives $h(\alpha,\beta) = (1+s)(\alpha+\beta) - \alpha^2 - \beta(\alpha + \beta)$. Then 
\begin{align*}
h(\alpha,\tfrac{1}{2} + s - \alpha) &= (1+s)(\tfrac{1}{2} + s) - \alpha^2 - (\tfrac{1}{2} + s - \alpha)(\tfrac{1}{2} + s) \\
h(\alpha,\tfrac{1}{2} + \tfrac{1}{2}s - \alpha) &= (1+s)(\tfrac{1}{2} + \tfrac{1}{2}s) - \alpha^2 - (\tfrac{1}{2} + \tfrac{1}{2}s - \alpha)(\tfrac{1}{2} + \tfrac{1}{2}s).
\end{align*}
Interpreting the above as functions in one variable, the second derivative test shows any critical point along the boundary is a local maximum.

All that remains is to consider the corner points; the smallest evaluation at a corner will be the minimum of $h(\alpha,\beta)$ in the given region. When $\delta(G)$ is even, the corner points are $(\frac{1}{2},\frac{1}{2}),(\frac{1}{2},s)$ and $(s,\frac{1}{2})$. When $\delta(G)$ is odd, the corner points are $(\frac{1}{2},\frac{1}{2}),(\frac{1}{2},s),(s,\frac{1}{2}),(s,\frac{1}{2}-\frac{1}{2}s)$ and $(\frac{1}{2}-\frac{1}{2}s,s)$. Given that $h(\alpha,\beta) = h(\beta,\alpha)$, it suffices to check the following three points.
\begin{align*}
h(\tfrac{1}{2},\tfrac{1}{2}) &= (1+s)\tfrac{1}{2} - \tfrac{1}{4} + (1+s)\tfrac{1}{2} - \tfrac{1}{4} - \tfrac{1}{4} = 1 + s - \tfrac{3}{4} = \tfrac{1}{4}+s \\
h(\tfrac{1}{2},s) &= (1+s)\tfrac{1}{2} - \tfrac{1}{4} + (1+s)s - s^2 - \tfrac{1}{2}s = \tfrac{1}{2} + \tfrac{1}{2}s - \tfrac{1}{4} + s + s^2 - s^2 - \tfrac{1}{2}s = \tfrac{1}{4}+s \\
h(\tfrac{1}{2} - \tfrac{1}{2}s,s) &= (1+s)(\tfrac{1}{2} - \tfrac{1}{2}s) - (\tfrac{1}{2} - \tfrac{1}{2}s)^2 + (1+s)s - s^2 - (\tfrac{1}{2} - \tfrac{1}{2}s)s \\
&= (\tfrac{1}{2} - \tfrac{1}{2}s)-(\tfrac{1}{2} - \tfrac{1}{2}s)^2 + s \\
&= (\tfrac{1}{2} - \tfrac{1}{2}s)(\tfrac{1}{2} + \tfrac{1}{2}s) + s \\
&= \tfrac{1}{4} - \tfrac{1}{4}s^2 + s.
\end{align*}
Since $s>0$, it follows that $h(\frac{1}{2}-\frac{1}{2}s,s) < h(\frac{1}{2},\frac{1}{2}),h(\frac{1}{2},s)$, which proves our result.

\section{Appendix C}
\label{sec:bipartworking}

Recall $\alpha_i = x_iy_i -z_i$ for $i=1,2,3$. Choose $x_1,y_1,z_1,x_2,y_2,z_2, x_3,y_3,z_3$ to maximise \begin{align}\label{eq:maxthis}\alpha_1 + \alpha_2 + \alpha_3 \end{align} subject to the following \emph{basic} constraints:
\begin{align*}
0 \leq x_i,y_i \,\,\forall i \qquad 0 \leq z_i \leq x_iy_i \,\,\forall i \qquad
x_1 + x_2 + x_3 =1 \qquad y_1 + y_2 + y_3=1
\end{align*}
and also the following \emph{balance} constraints:
\begin{align}
\alpha_1 \leq \alpha_2 + \alpha_3 \label{eq:balance1}\\ 
\alpha_2 \leq \alpha_3 + \alpha_1 \label{eq:balance2}\\ 
\alpha_3 \leq \alpha_1 + \alpha_2 \label{eq:balance3}
\end{align}
We prove that $\alpha_1 + \alpha_2 + \alpha_3 \leq \frac{1}{2}$.

\begin{claim}
\label{claim:no2ineq}
At most one of the balance constraints is a strict inequality.
\end{claim}
\begin{proof}
Assume for the sake of a contradiction that two of the balance constraints are strict inequalities, without loss of generality $\alpha_2 < \alpha_3 + \alpha_1$ and $\alpha_3 < \alpha_1 + \alpha_2$. Without loss of generality, $x_2 + y_2 \geq x_3 + y_3$. If $x_3 = 0$ then $\alpha_3 = x_3y_3 - z_3 \leq 0$, and so $\alpha_3 = 0$. Similarly, if $y_3=0$ then $\alpha_3=0$, and if $z_3 = x_3y_3$ then $\alpha_3 = 0$. However if $\alpha_3 = 0$ then the first two balance constraints give that $\alpha_1 \leq \alpha_2$ and $\alpha_2 \leq \alpha_1$. But this means that $\alpha_1 = \alpha_2$ and as such our assumption that $\alpha_2 < \alpha_3 + \alpha_1$ does not hold. Hence $x_3,y_3 > 0$ and $z_3 < x_3y_3$. %Define $\theta > 0$ such that $z_3 + \theta = x_3y_3$, and $\gamma>0$ such that $\alpha_2 + \gamma = \alpha_3 + \alpha_1$. 
Choose $\epsilon > 0$ such that $ \epsilon \leq x_3,y_3, \frac{x_3y_3 - z_3}{x_3+y_3}, \frac{\alpha_1+\alpha_3 - \alpha_2}{x_2 + y_2 + x_3 + y_3}$.

Define $x_2' = x_2+ \epsilon$, $y_2' = y_2 + \epsilon$, $x_3' = x_3 - \epsilon$ and $y_3' = y_3 - \epsilon$. We now show that by replacing $x_2$ with $x_2'$ and so on, we contradict our assumption that  $x_1,y_1,z_1, \dots, x_3,y_3,z_3$ maximise $\alpha_1 + \alpha_2 + \alpha_3$ with respect to all our constraints. 

First, check the basic constraints. By the choice of $\epsilon$, we have $x_3-\epsilon,y_3 - \epsilon \geq 0$. Also, $(x_3 - \epsilon)(y_3 - \epsilon) = x_3y_3 -\epsilon(x_3 + y_3) + \epsilon^2 \geq x_3y_3 - (x_3y_3 - z_3) +\epsilon^2 > z_3$, as required. All other basic constraints hold trivially.

Now we check the balance constraints. First consider \eqref{eq:balance1}. We prove this by contradiction. Suppose that $x_1y_1 - z_1 > x_2'y_2' - z_2 + x_3'y_3' - z_3.$ Thus
\begin{align*}
\alpha_1 = x_1y_1 - z_1 &> (x_2 + \epsilon)(y_2 + \epsilon) - z_2 + (x_3 - \epsilon)(y_3 - \epsilon) - z_3 \\
&= x_2y_2 - z_2 +x_3y_3 - z_3 + \epsilon(x_2 + y_2 + \epsilon - x_3 - y_3 + \epsilon) \\
&= \alpha_2 + \alpha_3 + \epsilon(x_2 + y_2 - x_3 - y_3 + 2\epsilon).
\end{align*} 
However, since $x_2 + y_2 \geq x_3 + y_3$, it follows that $\alpha_1 > \alpha_2 + \alpha_3$, which contradicts the fact that $x_1,y_1,z_1,\dots,x_3,y_3,z_3$ satisfy the balance constraints. To prove \eqref{eq:balance2}, suppose that $x_2'y_2' - z_2 > x_1y_1 - z_1 + x_3'y_3' - z_3.$
Thus
\begin{align*}
(x_2 + \epsilon)(y_2+\epsilon) - z_2 &> x_1y_1 - z_1 + (x_3 - \epsilon)(y_3 - \epsilon) - z_3 \\
x_2y_2 - z_2 + \epsilon(x_2 + y_2 + \epsilon) &> x_1y_1 - z_1 + x_3y_3-z_3 - \epsilon(x_3 + y_3 - \epsilon)\\
\alpha_2 + \epsilon(x_2 + y_2 + \epsilon) &> \alpha_1 + \alpha_3 - \epsilon(x_3 + y_3 - \epsilon) \\
\epsilon(x_2 + y_2 + x_3 + y_3) &> \alpha_1+\alpha_3 - \alpha_2.
\end{align*}
This contradicts our choice of $\epsilon$. Now consider \eqref{eq:balance3} and suppose that
$x_3'y_3' - z_3 > x_1y_1 - z_1 + x_2'y_2' - z_2.$
Thus \begin{align*}
x_3y_3-z_3 - \epsilon(x_3 + y_3 - \epsilon) &> x_1y_1-z_1 + x_2y_2 - z_2 + \epsilon(x_2 + y_2 + \epsilon) \\
\alpha_3 &> \alpha_1 + \alpha_2 + \epsilon(x_2+y_2+x_3+y_3).
\end{align*}
Since $\epsilon(x_2+y_2+x_3+y_3) \geq 0$, this again contradicts our choice of  $x_1,y_1,z_1,\dots,x_3,y_3,z_3$.

Finally, we now show that replacing $x_2$ with $x_2'$ and so on increases $\alpha_1 + \alpha_2 + \alpha_3$.
\begin{align*}
&\phantom{=}\,x_1y_1 - z_1 + x_2'y_2' - z_2 + x_3'y_3' - z_3 \\
&=\alpha_1 + \alpha_2 + \epsilon(x_2 + y_2 + \epsilon) + \alpha_3 - \epsilon(x_3 + y_3 - \epsilon) \\
&=\alpha_1 + \alpha_2 + \alpha_3 + \epsilon(x_2 + y_2 + \epsilon - x_3 - y_3 + \epsilon)
\end{align*}
This is a strict improvement since $x_2 + y_2 \geq x_3 + y_3$ and $2\epsilon > 0$.
\end{proof}

Thus, at least two of the balance constraints are equalities. Without loss of generality, $\alpha_1 = \alpha_2 + \alpha_3$ and $\alpha_2 = \alpha_3 + \alpha_1$. This forces $\alpha_3 = 0$.

If $z_1,z_2 >0$ then let $\epsilon = \min\{z_1,z_2\}$. If we replace $z_1,z_2$ with $z_1-\epsilon,z_2-\epsilon$ this maintains all constraints and increases $\alpha_1+\alpha_2+\alpha_3$. (We omit the proof of this as it is clear.) Thus without loss of generality $z_2=0$.

Now replace the balance constraints with the following two equivalent constraints:
\begin{align}
x_1y_1 - z_1 = x_2y_2 \label{eq:balance4} \\
x_3y_3 = z_3 \label{eq:balance5}
\end{align}

From this, it also follows that maximising \eqref{eq:maxthis} is equivalent to maximising $2x_2y_2.$

\begin{claim}
$z_1 = 0.$
\end{claim}
\begin{proof}
Assume that $z_1 > 0$. Also assume that $2x_2y_2 >0$ (for otherwise the entire result is proven). If $x_1 = 0$ or $y_1 = 0$, then $x_2y_2 = -z_1 < 0$, and so we may assume $x_1,y_1 > 0$. Choose $\epsilon>0$ such that $x_1-\epsilon,y_1-\epsilon,z_1 - 2\epsilon \geq 0$. As in Claim~\ref{claim:no2ineq}, we replace some choices of $x_1,y_1,z_1,\dots,x_3,y_3,z_3$ and show that our initial set of choices was not optimal.
Let $x_1' = x_1 - \epsilon, y_1' = y_1 -\epsilon, x_2' = x_2 + \epsilon, y_2' = y_2 + \epsilon, z_1' = z_1 - 2\epsilon.$ It is clear replacing $x_1$ with $x_1'$ and so on still satisfies the basic constraints, and increases $2x_2y_2$. The only difficult step is checking \eqref{eq:balance4}.
\begin{align*}
&\phantom{=}x_1'y_1' - z_1' \\
&= (x_1 - \epsilon)(y_1 - \epsilon) - z_1 + 2\epsilon \\
&= x_1y_1 -\epsilon(x_1+y_1) + \epsilon^2 - z_1 + 2\epsilon \\
&= x_1y_1 -z_1 -\epsilon(2-x_2 - y_2) + \epsilon^2+ 2\epsilon \\
&= x_2y_2 +\epsilon(x_2 + y_2) + \epsilon^2 \\
&= (x_2+\epsilon)(y_2 + \epsilon) \\
&= x_2'y_2'
\end{align*}
Hence \eqref{eq:balance4} still holds, and thus our choice of $x_1,y_1,z_1,\dots,x_3,y_3,z_3$ was not optimal, a contradiction.
\end{proof}

Thus $x_1y_1 = x_2y_2$. Define $c,d \in [-\frac{1}{2},\frac{1}{2}]$ such that $x_2 = \frac{1}{2}+c$ and $y_2 = \frac{1}{2}+d$. Thus $x_1 \leq \frac{1}{2}-c, y_1 \leq \frac{1}{2}-d$. Hence $(\frac{1}{2}-c)(\frac{1}{2}-d) \geq (\frac{1}{2}+c)(\frac{1}{2}+d)$, and so $c \leq -d$. Finally, this means that $$\alpha_1 + \alpha_2 + \alpha_3 = 2x_2y_2 = 2(\tfrac{1}{2}+c)(\tfrac{1}{2}+d) \leq 2(\tfrac{1}{2}-d)(\tfrac{1}{2}+d) = \tfrac{1}{2} - 2d^2 \leq \tfrac{1}{2}.$$

\bibliographystyle{abbrvnat}
\bibliography{genlinebib}

\end{document}